\documentclass[11pt]{amsart}
\usepackage[utf8]{inputenc}
\usepackage[T1]{fontenc}
\usepackage{newtxtext,newtxmath}
\usepackage{microtype}
\usepackage[margin=.85in]{geometry}
\usepackage[table]{xcolor} 
\usepackage{verbatim}
\usepackage{multirow}
\usepackage{graphicx}
\usepackage{enumerate}
\usepackage{ytableau}
\usepackage{tikz}
\usepackage{tikz-cd}
\usepackage{float,pifont}
\usepackage{mathtools}
\usepackage[pagebackref,colorlinks=true,citecolor=blue]{hyperref}
\usepackage[capitalise]{cleveref}
\usetikzlibrary{matrix,decorations.pathreplacing}

\usepackage{parskip}

\newtheorem{theorem}{Theorem}[section]
\newtheorem{lemma}[theorem]{Lemma}
\newtheorem{corollary}[theorem]{Corollary}

\newtheorem{claim}[theorem]{Claim}
\theoremstyle{definition}
\newtheorem{definition}[theorem]{Definition}
\newtheorem{notation}[theorem]{Notation}

\newtheorem{remark}[theorem]{Remark}

\DeclareMathOperator{\reg}{\rm{reg}}
\DeclareMathOperator{\wt}{\rm{wt}}

\usepackage{ascii}
\usepackage{setspace} 
\usepackage{listings}
\usepackage{fancyvrb}

\title[Projective monomial curves]{Projective monomial curves associated to numerical semigroups with multiplicity $e$, width $e-1$, and embedding dimension $e-2$}

\author{Om Prakash Bhardwaj}
\address{Chennai Mathematical Institute, Siruseri, Tamilnadu~603103. India}
\email{omprakash@cmi.ac.in, opbhardwaj95@gmail.com}

\author{Trung Chau}
\address{Chennai Mathematical Institute, Siruseri, Tamilnadu~603103. India}
\email{chauchitrung1996@gmail.com}

\author{Omkar Javadekar}
\address{Chennai Mathematical Institute, Siruseri, Tamilnadu~603103. India}
\email{omkarj@cmi.ac.in, omkarjavadekar@gmail.com}

\keywords{Numerical Semigroup, Projective Monomial Curve, Gr\"obner Basis, Castelnuovo--Mumford Regularity}

\subjclass[2020]{13P10, 13F55, 13D02, 20M14}

\begin{document}

\begin{abstract}
    Numerical semigroups with multiplicity $e$, width $e-1$, and embedding dimension $e-2$ are of the form
    $$S(e,m,n) = \langle \{e, e+1, \ldots, 2e-1\} \setminus \{e+m, e+n\} \rangle,$$
for some $1 \leq m < n \leq e-2$. 
Inspired by the work of Sally, Herzog and Stamate studied the special case $S(e,2,3)$, which they called the ``Sally numerical semigroups''. Recently, Dubey  et. al. computed a minimal generating set of the defining ideal of the numerical semigroups $S(e,m,n)$ for $m \geq 2$. In this article, we first obtain an analog for the numerical semigroups $S(e,1,n)$, and then shift our focus to the projective monomial curves in $\mathbb{P}^{e-2}$ defined by the semigroups $S(e,m,n)$. 
We obtain a Gr\"{o}bner basis for the defining ideal of the projective monomial curves associated to the semigroups $S(e,m,n)$. Moreover, we provide characterizations of Cohen--Macaulay and Gorenstein properties of these curves. Specifically, we prove that these are Cohen--Macaulay if and only if $(m,n) \neq (e-4,e-3)$, and Gorenstein if and only if $(e,m,n)\in \{ (4,1,2), (5,2,3)\}$. Furthermore, when these curves are Cohen--Macaulay, we compute the Castelnuovo--Mumford regularity of their coordinate ring.\\

\end{abstract}

\maketitle

\section{Introduction}

{{Let $S$ be a numerical semigroup with minimal generators $a_0 < a_1 < \cdots < a_p$. Let $\mathbb{K}$ be a field. Consider the $\mathbb{K}$-algebra homomorphism $\phi: \mathbb{K}[x_0, \ldots, x_p] \longrightarrow \mathbb{K}[S] \coloneqq \mathbb{K}[t^{a_0}, \ldots, t^{a_p}]$ where $\phi(x_i) = t^{a_i}$ for all $i$. The kernel of $\phi$ is the defining ideal of the affine curve $\mathcal{C}_S$ in $\mathbb{A}_{\mathbb{K}}^{p+1}$ parametrized by $(t^{a_0}, t^{a_1}, \ldots, t^{a_p})$. The projective closure of the affine curve $\mathcal{C}_S$ is the projective monomial curve in projective space $\mathbb{P}_{\mathbb{K}}^{p+1}$ defined parametrically by 
$$
x_0 = t_1^{a_0}t_2^{a_p-a_0}; \quad x_1 = t_1^{a_1}t_2^{a_p-a_1}; \quad  \ldots \quad ; \quad  x_{p-1} = t_1^{a_{p-1}}t_2^{a_p - a_{p-1}}; \quad x_p = t_1^{a_p}; \quad y = t_2^{a_p},
$$
where $[x_0:x_1:\cdots:x_p:y]$ denote the homogeneous coordinates in $\mathbb{P}_{\mathbb{K}}^{p+1}$. The affine monomial curves are always Cohen--Macaulay, since their coordinate ring is a one-dimensional domain. There is also a well-known criterion for the Gorensteinness of affine monomial curves due to Kunz~\cite{Kunz}. Specifically, he showed that Gorensteinness is equivalent to the numerical semigroup being symmetric. However, the situation for projective monomial curves is not so straightforward. It is not always true that the projective closure of an affine monomial curve is Cohen--Macaulay. 
The first example of a non--Cohen--Macaulay projective monomial curve was given by Macaulay~\cite{Macaulay}, parametrized by $(t_1^3t_2, t_1t_2^3, t_1^4, t_2^4)$.

Over the past few decades, there has been substantial progress in characterizing important properties such as the Cohen--Macaulayness, Gorensteinness, and complete intersection nature of projective monomial curves. Alongside these structural properties, invariants such as the Cohen--Macaulay type and the Castelnuovo--Mumford regularity have been studied extensively (see, e.g., \cite{Bermejo-arithmetic, Bermejo-CM-reg, Bermejo-Giminez-reg, Bermejo-CM-reg-in-codim-2, Bresinsky-Schenzel-Vogel, Morales, Hoa-Trung, Goto, Herzog-Hibi-reg, Herzog-Stamate-CM, Reid-2, Reid-1, Pranjal-Sengupta}). In this article, we consider the projective monomial curves associated with numerical semigroups having multiplicity $e$, width $e-1$, and embedding dimension $e-2$. The multiplicity is the smallest non-zero integer in the semigroup, the width is the difference between the largest and the smallest minimal generators, and the embedding dimension is the number of minimal generators of the numerical semigroup. In other words, the numerical semigroups having multiplicity $e$, width $e-1$, and embedding dimension $e-2$ are exactly of the form
$$S(e,m,n) = \langle \{e, e+1, \ldots, 2e-1\} \setminus \{e+m, e+n\} \rangle,$$
for some $1 \leq m < n \leq e-2$. 

{{In 1980, Sally~\cite{Sally} proved that the associated graded ring of a one-dimensional Gorenstein local ring with multiplicity $e$ and embedding dimension $e-2$ is Cohen--Macaulay. Herzog and Stamate~\cite{Herzog-Stamate} later showed that the numerical semigroup $S(e,2,3)$ defines a Gorenstein semigroup ring satisfying Sally's conditions; such semigroups are called \textit{Sally semigroups}. More recently, Dubey et. al. initiate the study of $S(e,m,n)$ for $m\geq 2$, which they call \emph{numerical semigroups of Sally type}. Specifically, they determined a minimal generating set for the defining ideal of the corresponding semigroup rings and characterized their Gorenstein property. In this article, we extend these results to the cases where $m=1$. In addition, we study the homogenization of these defining ideals, which correspond to projective monomial curves, together with a characterization for their Cohen--Macaulay and Gorenstein properties (we refer to \cref{sec:prelim} for unexplained terminology). Our technique relies on finding the Gr\"obner bases of these defining ideals, the Cohen--Macaulay criterion of projective monomial curves by Herzog and Stamate (\cite[Proposition~3.15]{Herzog-Stamate-CM}), and the well-known fact that the Hilbert function of Gorenstein Artinian standard graded $\mathbb K$-algebras is symmetric (\cite[Theorem 4.1]{Stanley}). We also compute the Castelnuovo--Mumford of these homogenization using a formula for the regularity of finite-length modules over standard graded polynomial rings (\cite[Exercise 20.18]{Eis-book}). 
}}

This article is structured as follows. In \Cref{sec:prelim} we set up the notation, and collect the definitions and known results that are used in the rest of the article. \Cref{sec:m=1-case} focuses on the numerical semigroups of the form $S(e,1,n)$. We provide the generating set for the defining ideal of such semigroups (see \Cref{thm:gens-for-m=1}), and characterize the Gorenstein ones among them (see \Cref{thm:gor-for-m=1}). In \Cref{sec:GB-CM}, we compute the Gr\"obner basis of the defining ideal of $\overline{S(e,m,n)}$ for all $e\geq 10$ by computing the Gr\"obner basis for the defining ideal of $S(e,m,n)$ (see \Cref{thm:GB-proof-general} and \Cref{cor:GB-for-projective}). With the help of the Gr\"obner basis obtained, we characterize the Cohen--Macaulay and Gorenstein properties of these curves (see \Cref{cor:CM-characterization} and \Cref{thm:Gor}). Finally, when the projective monomial curves defined by $S(e,m,n)$ are Cohen-Macaulay, we  calculate the Castelnuovo--Mumford regularity their coordinated ring (see \Cref{cor:CM-characterization}). 

\section*{Acknowledgement}
The authors would like to thank Prof.~Hema Srinivasan for useful discussions on Sally numerical semigroups. The three authors are partially supported by the Infosys Foundation.

\section{Preliminaries}\label{sec:prelim}
Throughout this article $\mathbb K$ denotes an algebraically closed field of characteristic zero, and $\mathbb N$ is the set of all non-negative integers.

\begin{definition}
Let $S$ be a submonoid of $\mathbb{N}$ such that $\mathbb{N} \setminus S$ is finite. Then $S$ is called a \textit{numerical semigroup}. Equivalently, there exist $a_0,\ldots,a_p \in \mathbb{N}$ such that $\mathrm{gcd}(a_0,\ldots,a_p) = 1$ and 
\[
S =  \langle a_0, \ldots, a_p \rangle = \left\lbrace \sum_{i=0}^{p} \lambda_i a_i \mid \lambda_i \in \mathbb{N}, \text{\ for\  all\ } i \right\rbrace.
\]
\end{definition}

Every numerical semigroup $S$ has a unique minimal generating set. The cardinality of the minimal generating set of $S$ is called the \textit{embedding dimension of $S$}. The smallest non-zero element inside $S$ is called the \textit{multiplicity} of $S$, and the difference between the largest and the smallest minimal generator is called the \textit{width} of $S$, and it is denoted by $\operatorname{wd}(S)$.
Since $\mathbb{N}\setminus S$ is finite,  $\mathbb{N}\setminus S$ has a maximum, called the \textit{Frobenius number} of $S$, and it is denoted by $F(S)$. A numerical semigroup $S$ is called \textit{symmetric} if for all $s \in \mathbb{Z}\setminus S$ we have $F(S)-s \in S$.

 Let ${\mathbf{a}}_0,{\mathbf{a}}_1, \ldots , {\mathbf{a}}_p \in \mathbb{N}^r$ then
\[S = \langle {\mathbf{a}}_0,{\mathbf{a}}_1, \ldots ,{\mathbf{a}}_p \rangle = \left\lbrace \sum_{i=0}^{p} \lambda_i {\mathbf{a}}_i \mid \lambda_i \in \mathbb{N}, \text{\ for\  all\ } i \right\rbrace\]
is called an \textit{affine semigroup} generated by  ${\mathbf{a}}_0,{\mathbf{a}}_1, \ldots,  {\mathbf{a}}_p$. Every affine semigroup $S$ has a unique minimal generating set. For $r=1$, affine semigroups correspond to numerical semigroups. Let $\mathbb K$ be a field, the semigroup ring $\mathbb K[S]:=\oplus_{s \in S} \mathbb K {\bf t}^s $ of $S$ is a $\mathbb K$-subalgebra of the polynomial ring $\mathbb K[t_1,\ldots,t_r]$, where $t_1,\ldots,t_r$ are indeterminates and ${\bf t}^s = \prod_{i=1}^r t_i^{s_i}$, for all $s = (s_1,\ldots,s_r) \in S$. The semigroup ring $\mathbb K[S] = \mathbb K[{\bf t}^{{\mathbf{a}}_0},{\bf t}^{{\mathbf{a}}_1}, \ldots, {\bf t}^{{\mathbf{a}}_p}]$ of $S$ can be represented as a quotient of a polynomial ring using a canonical surjection $\pi : \mathbb K[x_0,x_1,\ldots,x_p] \rightarrow \mathbb K[S]$, given by $\pi(x_i) = {\bf t}^{{\mathbf{a}}_i}$ for all $i=0,1,\ldots,p.$ The kernel of the map $\pi$ is called the \textit{defining ideal} of $S$, and it is denoted by $I_S$. Define the \textit{weight of $x_i$} as $\wt(x_i) = {\mathbf{a}}_i$ for $0 \leq i \leq p$. 
For a monomial ${\bf{x}^u}:= x_1^{u_1}x_2^{u_2}\cdots x_n^{u_n}$, define $\wt( {\bf{x}^u}) = \sum_{i=1}^n u_i{\mathbf{a}}_i$. By a well known fact (e.g. see \cite[Theorem 7.3]{Miller-Sturmfels}), we have 
$$
I_S = ({\bf{x}^u}-{\bf{x}^v} \mid \wt ({\bf{x}^u}) = \wt({\bf{x}^v})).
$$
Thus, here $I_S$ is a binomial prime ideal (toric ideal) of $\mathbb{K}[x_0,\ldots,x_p]$ which is homogeneous with respect to the grading defined by $\wt(-)$. Note that the weight function $\wt(-)$ depends on the semigroup $S$. However, in our work, since the semigroup under consideration will always be clear from the context, for the sake of brevity, we shall not mention the semigroup $S$ eplicitly while talking out $\wt(-)$. When $S$ is a numerical semigroup, i.e., when $r=1$, we have the following result due to Gastinger. 

\begin{theorem}[see {\cite[Theorem 4.8]{Eto}}]\label{thm:dim}
    Let $S$ be a numerical semigroup minimally generated by $a_0, \ldots, a_p$, and $J \subseteq I_S$ be an ideal of $\mathbb K[x_0, \ldots, x_p]$. Then $J = I_S$ if and only if
    $$\mathrm{dim}_{\mathbb{K}} \left(\frac{\mathbb K[x_0, \ldots, x_p]}{J + (x_i)}\right)  = a_i.$$
\end{theorem}

\medskip

Let $S$ be a numerical semigroup minimally generated by $a_0 < a_1 < \cdots < a_p$. Let $\mathcal{C}$ be a projective monomial curve in the projective space $\mathbb{P}_K^{p+1}$, defined parametrically by
\begin{center}
$x_0 = t_1^{a_0}t_2^{a_p-a_0}; \quad x_1 = t_1^{a_1}t_2^{a_p-a_1}; \quad  \ldots \quad ; \quad  x_{p-1} = t_1^{a_{p-1}}t_2^{a_p - a_{p-1}}; x_p = t_1^{a_p}; \quad y = t_2^{a_p}$.
\end{center}

We call $\mathcal{C}$ the \textit{projective monomial curve defined by $S$}. Let $\mathbb K[\mathcal{C}]$ denote the coordinate ring of $\mathcal{C}$. Then 
$\mathbb K[\mathcal{C}] \cong \mathbb K[\bar{S}] = \dfrac{\mathbb{K}[x_0, \ldots, x_p, y]}{I_{\bar S}}$, where $\bar{S} = \langle  (a_0,a_p-a_0), (a_1,a_p- a_1),\ldots, (a_{p-1},a_p-a_{p-1}), (a_p,0), (0,a_p) \rangle$ is an affine semigroup in $\mathbb{N}^2$. The defining ideal $I_{\bar S}$ of $\bar S$ is in fact the homogenization of $I_S$ with respect to the variable $y$. 

It is important to note that if $I_S=(f_1, 
\ldots, f_r)$, then $I_{\bar S}$ need not be generated by the homogenizations $f_1^h, \ldots, f_r^h$ of $f_1, \ldots, f_r$. However, if $f_1, \ldots, f_r$  is a Gr\"obner basis of $I_S$ with respect to the graded reverse lexicographic order on $\mathbb K[x_0, \ldots, x_p]$, then $f_1^h, \ldots, f_r^h$ is a Gr\"obner basis of $I_{\bar S}$ with respect to the graded reverse lexicographic order on  $\mathbb K[x_0, \ldots, x_p, y]$ (see \cite[Proposition 3.15]{Ene-Herzog}). 

Herzog and Stamate \cite{Herzog-Stamate-CM} provided the following criteria for Cohen--Macaulayness of projective monomial curves in terms of Gr\"obner basis, which we shall crucially use throughout this article.
\begin{theorem}[cf.~{\cite[Theorem 2.2]{Herzog-Stamate-CM}}]\label{thm:CM-criterion}
    Let  $a_0 < a_1 < \cdots < a_p$  be a sequence of positive integers, $S$ be the numerical semigroup minimally generated by them. Let $<$ be the graded reverse lexicographic order on $\mathbb K[x_0, \ldots, x_p ]$, and $<'$ be the extended graded reverse lexicographic order on $\mathbb K[x_0, \ldots, x_p, y ]$, then the following conditions are equivalent:
    \begin{enumerate}[{\rm (i)}]
        \item $\mathbb K[\bar S]$ is Cohen--Macaulay.
        \item $x_p$ does not divide any element of the minimal monomial generating set of $\operatorname{in}_{<}(I_S)$.
        \item $x_p$ and $y$ do not divide any element of the minimal monomial generating set of $\operatorname{in}_{<'}(I_{\bar S})$.
    \end{enumerate}
\end{theorem}
\medskip

We now set up the notation for the rest of the article. 

\begin{notation} \hfill{}
    \begin{enumerate}[{\rm (a)}]
    \item For $1\leq m<n\leq e-2$, $S(e,m,n)\coloneqq \langle \{ e, e+1, \ldots, 2e-1\} \setminus \{e+m, e+n\} \rangle$.
    \item $R(e,m,n)\coloneqq \mathbb{K}[\{x_0, \ldots, x_{e-1}\}\setminus \{x_m, x_n\}]$.
    \item $R'(e,m,n)\coloneqq\mathbb K\left[\{x_0,\ldots, e-2\} \setminus\{m,n\} \right]$
    \item $\mathbb K[S(e,m,n)]\coloneqq \dfrac{R(e,m,n)}{I_{S(e,m,n)}}$.
    \item $\overline{S(e,m,n)}\coloneqq \langle \{(e+i, (2e-1)-(e+i) \mid i \in \{0, \ldots, e-1\} \setminus\{m,n\}\}, (0,2e-1)\rangle$
    \item $\mathbb K\left[\overline{S(e,m,n)}\right]\coloneqq \dfrac{R(e,m,n)[y]}{I_{\overline{S(e,m,n)}}}$.    
    \end{enumerate}
\end{notation}

\section{Numerical semigroups of the form \texorpdfstring{$S(e,1,n)$}{S(e,1,n)}}\label{sec:m=1-case}

This section focuses on the numerical semigroups of the form $S(e,m,n)$ with $m=1$. To study the case of projective monomial curves associated to $\overline{S(e,m,n)}$, we need a complete description of the defining ideal of the corresponding numerical semigroup $S(e,m,n)$. In a recent work~\cite{Hema-Sallytype}, the authors determined minimal generating sets 
for these ideals in the case $m \geq 2$. However, the case $m = 1$ was not considered. 
Therefore, we begin by deriving the minimal generating set for the defining ideal $S(e, m, n)$ when $m = 1$.

For $n \in [2,e-2]$, we define the following matrices.

$$A_n = \begin{bmatrix}
 x_2 & x_3 & \cdots & x_{n-2} & x_{n+1} & x_{n+2} & \cdots & x_{e-2} & x_{e-1} \\

 x_3 & x_4 & \cdots & x_{n-1} & x_{n+2} & x_{n+3} & \cdots & x_{e-1} & x_0^2
\end{bmatrix}, \quad B_2 = \begin{bmatrix}
    x_0 & x_3 & \cdots & x_{e-4} & x_{e-3} \\

 x_3 & x_6 & \cdots & x_{e-1} & x_0^2
\end{bmatrix} $$

$$B_n = \begin{bmatrix}
 x_0 & x_2 & \cdots & x_{n-3} & x_{n-1} & x_{n+1} & \cdots & x_{e-3} & x_{e-2} \\

 x_2 & x_4 & \cdots & x_{n-1} & x_{n+1} & x_{n+3} & \cdots & x_{e-1} & x_0^2
\end{bmatrix};\ \text{\ for\ } n \geq 3 $$

Let $J_n$ denote the set of all $2 \times 2$ minors of $A_n$, and let $K_n$ denote the set of all $2\times 2$ minors of $B_n$ with the first column. 

\begin{theorem}\label{thm:gens-for-m=1}
    { Let $e \geq 10$ and $n \in \{2, \ldots, e-2\}$. Then a minimal generating set $G_{S(e,1,n)}$ of $I_{S(e,1,n)}$ is given by $G_{S(e,1,n)} = J_n \cup K_n \cup L_{e,n}$, where

    \begin{enumerate}[{\rm (1)}]
        \item $L_{e,2} = \emptyset$.
        
        \item $L_{e,3} = \{x_4x_{4+i} - x_2x_{6+i} \mid i \in [0,e-7]\} \ \cup \ \{x_0^2x_2 - x_4x_{e-2}\}$.
        

  \item  $L_{e,4} = \{x_5x_{5+i}-x_3x_{7+i} \mid 0 \leq i \leq e-8\}\  \cup  \{x_3^2-x_0x_{6}\}  \cup \left\{\begin{array}{c} x_0^2x_3-x_5x_{e-2}, x_0x_2^2-x_5x_{e-1}, x_2^3-x_7x_{e-1} \end{array}\right\}$

\item For $5 \leq n \leq e-4$,
  $L_{e,n} = \begin{array}{l} \{x_{n+1}x_{n+1+i} - x_{n-1}x_{n+3+i} \mid i \in [0,e-n-4]\} \ \cup \\ \{x_{n-1}x_{n-1-i} - x_{n+1}x_{n-3-i} \mid [0,n-5]\} \ \cup \ \{x_{n-1}x_3 - x_{n+2}x_0\} \cup \\ \{x_0^2x_{n-1}- x_{n+1}x_{e-2}, x_0x_2x_{n-2}-x_{n+1}x_{e-1}\} \end{array}$

\item $L_{e,e-3} = \{x_{e-4}x_{e-4-i} - x_{e-2}x_{e-6-i} \mid i \in [0,e-8]\} \cup \{x_{e-4}x_3 - x_0x_{e-1}\} \cup \{x_0^2x_{e-4}- x_{e-2}^2, x_0x_2x_{e-5}-x_{e-2}x_{e-1}\}$
        

        \item $L_{e,e-2} = \{x_{e-3}x_{e-3-i}-x_{e-1}x_{e-5-i} \mid  i \in [0, e-7]\} \cup \{x_0^3-x_3x_{e-3}, x_0x_2x_{e-4}-x_{e-1}^2\}.$
    \end{enumerate}}
\end{theorem}
\begin{proof}
   Given integers $e$ and $n$, let $I_{e,n}$ be the ideal generated by $J_n\cup K_n \cup L_{e,n}$. Observe that $I_{e,n}$ is minimally generated by $J_n\cup K_n \cup L_{e,n}$ (as the initial terms of the polynomials in this generating set do not divide one another), and we also have $I_{e,n}\subseteq I_{S(e,1,n)}$. We show that $\dim_{\mathbb{K}} \left(\dfrac{R(e,1,n)}{I_{e,n}+(x_0)}\right) = e$, and use \Cref{thm:dim} to conclude the proof. Since $I_{S(e,1,n)}$ is generated in degree $\geq 2$, the set  $\{1, x_2,x_3, \ldots, x_{n-1}, x_{n+1}, x_{n+2}, \ldots, x_{e-1}\}$ consisting of $e-2$ elements forms a part of the monomial $\mathbb{K}$-basis of $R/(I_{S(e,1,n)}+(x_0))$, and hence of $R(e,1,n)/(I_{e,n}+(x_0))$. In each of the cases (1)-(6), we prove that all but two monomials of degree 2 are nonzero in $R(e,1,n)/(I_{e,n}+(x_0))$. 
Consider the reverse graded lexicographic order $<$ on $R(e,1,n)$ 
. We show that $\operatorname{in}_<(I_{e,n}+(x_0))$ contains all but two monomials of degree $2$ and all monomials of degree $\geq 3$. To do this, we first observe the following: 
\begin{enumerate}[(a)]
    \item Every monomial divisible by $x_0$ is in $\operatorname{in}_<(I_{e,n}+(x_0))$.
    \item Let $i>0$. Then $x_ix_{e-1}\in \operatorname{in}_<(I_{e,n}+(x_0))$ for all $i \neq 2, n+1$. This is because from the matrix $A_n$, we see that  for such $i$, we have $x_ix_{e-1}-x_0^2x_{i-1} \in I_{e,n}$.\\
    Also, for $n \neq 2,3$, we have $x_{n+1}x_{e-1} \in \operatorname{in}_<(I_{e,n})+(x_0))$, since in this case, $x_0x_2x_{n-2}-x_{n+1}x_{e-1} \in I_{e,n}$.

    \item Assume $i \neq 0, e-1$. Then, for $n\neq 4$, the monomial $x_i^2 \in \operatorname{in}_<(I_{e,n}+(x_0))$ for all $i$. This is evident for $i\neq 2,n-1, n+1$ from $A_n$, for $i=2$ from $B_n$, for $i=n-1, n+1$ from the set $L_{e,n}$.\\
    Similarly, for $n=4$, $x_i^2\in \operatorname{in}_<(I_{e,n}+(x_0))$ for all $i \neq 2$. Furthermore, when $n=4$, we have $x_2^3\in \operatorname{in}_<(I_{e,n}+(x_0))$ because $x_2^3- x_7x_{e-1} \in I_{e,n}$.
    \item $x_2x_j \in \operatorname{in}_<(I_{e,n} + (x_0))$ for $2<j\leq e-2$ with $j \neq n-2$, as evident from $B_n$.\\
    For $3\leq i <j\leq e-2$, we have $x_ix_j \in \operatorname{in}_< (I_{e,n} + (x_0))$. This is evident from $A_n$ for $i \neq n+1$ and for $j\neq n-1$, and from $L_{e,n}$ for $i=n+1$ and for $j=n-1$.
\end{enumerate} 

We are now ready to show that all monomials of degree $3$ are in $\operatorname{in}_{<}(I_{e,n}+(x_0))$. There are three types of monomials of degree 3, namely (i) $x_i^3$, (ii) $x_i^2x_j$, with $i\neq j$, and (iii) $x_ix_jx_k$ with $i<j<k$.

{Case (i):} From (a), (b), (c) above, we see that $x_i^3 \in \operatorname{in}_<(I_{e,n}+(x_0))$ for all $i$. 

{Case (ii):} From (a)-(d), we see that $x_i^2x_j \in \operatorname{in}_<(I_{e,n}+(x_0))$ except possibly for the case $n=4, i=2, j=e-1$. But $x_2^2x_{e-1} \in \operatorname{in}_<(I_{e,n}+(x_0))$, since $x_2^2x_{e-1} = x_2(x_2x_{e-1}-x_{3}x_{e-2}) - x_3(x_0^3-x_2x_{e-2}) + x_0^3x_3$, with $x_2x_{e-1}-x_{3}x_{e-2}, x_0^3-x_2x_{e-2} \in I_{e,4}$. 

{Case (iii):} From (a), we have $x_0x_jx_k \in \operatorname{in}_<(I_{e,n}+(x_0))$. Assume $i>0$. From (d), we see that $x_ix_jx_k \in \operatorname{in}_<(I_{e,n}+(x_0))$ for $k\leq e-2$. Finally, from (b) and (d), we see that $x_ix_jx_k \in \operatorname{in}_<(I_{e,n}+(x_0))$ for $k=e-1$. 

This proves that all degree 3 monomials are in $\operatorname{in}_<(I_{e,n}+(x_0))$.  We now consider degree 2 monomials. In each case, using observations (a)-(d), in every case it follows that except possibly two monomials, all other monomials belong to $\operatorname{in}_<(I_{e,n}+(x_0))$. More precisely, we have the following.



\begin{enumerate}
    \item  For $n=2$, except $x_3x_{e-2}$ and $x_3x_{e-1}$, all degree $2$ monomials are in $\operatorname{in}_<(I_{e,2}+(x_0))$.

\item  For $n=3$, except $x_2x_{e-1}$ and $x_{4}x_{e-1}$, all degree $2$ monomials are in $\operatorname{in}_<(I_{e,3}+(x_0))$.

\item  For $n=4$, except $x_2x_{n-2}$ and $x_{2}x_{e-1}$, all degree $2$ monomials are in $\operatorname{in}_<(I_{e,4}+(x_0))$.   

\item  For $5 \leq n \leq e-4$, except $x_2x_{n-2}$ and $x_{2}x_{e-1}$, all degree $2$ monomials are in $\operatorname{in}_<(I_{e,n}+(x_0))$. 

\item  For $n=e-3$, except $x_2x_{n-2}$ and $x_{2}x_{e-1}$, all degree $2$ monomials are in $\operatorname{in}_<(I_{e,e-3}+(x_0))$.

\item  For $n=e-2$, except $x_2x_{n-2}$ and $x_{2}x_{e-1}$, all degree $2$ monomials are in $\operatorname{in}_<(I_{e,e-2}+(x_0))$.
\end{enumerate}

This shows that, in each case $\dim_{\mathbb{K}} \left(\dfrac{R}{I_{e,n}+(x_0)}\right)\leq e$. Since $I_{e,n}+(x_0)\subseteq I_{S(e,1,n)}$, we already have $\dim_{\mathbb{K}}\left(\dfrac{R}{I_{e,n}+(x_0)}\right)\geq  e$.  This completes the proof. 
\end{proof}

\begin{corollary}\label{cor:number-mingens}
    Let $e \geq 10$. Then, the cardinality of the minimal generating set of $I_{S(e,1,n)}$ is given by
  \[
\mu(I_{S(e,1,n)}) =
\begin{cases}
\binom{e-2}{2} - 1 & \text{if } n = 4, \\[6pt]
\binom{e-2}{2} - 2 & \text{otherwise.}
\end{cases}
\]

\end{corollary}

\begin{remark}
    In \cite[Conjecture 2.1]{Herzog-Stamate}, Herzog and Stamate conjectured that for any numerical semigroup $S$, $\mu(I_S) \leq \binom{\operatorname{wd}(S)+1}{2}$. By \Cref{cor:number-mingens}, it follows that the numerical semigroups of the form $S(e,1,n)$ satisfy this conjecture.   
\end{remark}

It is well-known that numerical semigroup rings  are 1-dimensional domains, and in particular Cohen--Macaulay. We proceed to characterize the Gorenstein property for $\mathbb K[S(e,m,n)]$, extending \cite[Theorem 2.3]{Hema-Sallytype}.

\begin{theorem}\label{thm:gor-for-m=1}
 The semigroup ring $\mathbb{K}[S(e,m,n)]$ is  Gorenstein if and only if $(m,n) = (2,3)$ or $(e,m,n) = (4,1,2)$. 
\end{theorem}
\begin{proof}
    When $m \neq 1$, the result is proved in \cite[Theorem 2.3]{Hema-Sallytype}. Therefore, we assume that $m = 1$. By the work of Kunz \cite{Kunz}, it is well-known that a numerical semigroup ring is Gorenstein if and only if its value semigroup is symmetric. 

    Let $e=4$. Then $n=2$, $S(4,1,2)=\langle 4,7\rangle$, and we have $I_{S(e,1,4)}=( x_0^{7}-x_{3}^4)$. Hence $\mathbb{K}[S(4,1,2)]$ is Gorenstein. 

    Now, assume that $e \geq 5$. It suffices to show that given any $2 \leq n \leq e-2$, the semigroup $S(e,1,n)$ is not symmetric. 
    
    When $n=2$, we have $F(S(e,1,n))= 2e+2$. In this case, we see that $1 \not\in S(e,1,n)$ as well as $(2e+2)-1 =e+1\not\in S(e,1,n)$. Hence $S(e,1,n)$ is not  symmetric.
      
    When $n=3$, we have $F(S(e,1,n))= 2e+3$. In this case, we see that $2 \not\in S(e,1,n)$ as well as $(2e+3)-2 =e+1\not\in S(e,1,n)$. Hence $S(e,1,n)$ is not  symmetric.

    When $n\geq 4$, we have $F(S(e,1,n))= 2e+1$. In this case, we see that $e+n \not\in S(e,1,n)$ as well as $(2e+1)-(e+n) =e-n+1\not\in S(e,1,n)$. Hence $S(e,1,n)$ is not  symmetric.
\end{proof}

\section{Gr\"obner basis and Cohen-Macaulayness}\label{sec:GB-CM}

The aim of this section is to characterize the Cohen--Macaulay and Gorenstein properties for $\mathbb K\left[\overline{S(e,m,n)}\right]$. We achieve this with the help of the Gr\"obner basis computation of the defining ideal $I_{\overline{S(e,m,n)}}$. Recall that, if $f_1, \ldots, f_r$  is a Gr\"obner basis of $I_{S(e,m,n)}$ with respect to the graded reverse lexicographic order, then $f_1^h, \ldots, f_r^h$ is a Gr\"obner basis of $I_{\overline{S(e,m,n)}}$ with respect to the graded reverse lexicographic order considering $y$ as the smallest variable. Therefore, it is sufficient to compute the Gr\"obner basis of $I_{S(e,m,n)}$. 

For $m=1$, a generating set of $I_{S(e,m,n)}$ is given in \Cref{thm:gens-for-m=1}. 
{{For $m \geq 2$, a minimal generating set for $I_{S(e,m,n)}$ is obtained in \cite{Hema-Sallytype}. For our purposes, we modify their set so that it works well with the graded reverse lexicographic order $<$ that we shall consider in the subsequent part of the article. In the original generating set, some elements have the same initial terms with respect to $<$. The modified generating set, which we record below, removes these overlaps and makes the set more suitable for constructing a Gr\"obner basis.}}



\begin{theorem}[{\cite[cf. Theorem 5.2]{Hema-Sallytype}}]\label{thm:gens-for-m-geq-2}
{Let $e \geq 10$ and $m \geq 2$. Define 

\[
A_{m,n} =
\left[
\begin{array}{ccccccccccccc}
x_0 & x_1 & \cdots & x_{m-2} & x_{m+1} & x_{m+2} & \cdots & x_{n-2} & x_{n+1} & x_{n+2} & \cdots & x_{e-2} & x_{e-1} \\
x_1 & x_2 & \cdots & x_{m-1} & x_{m+2} & x_{m+3} & \cdots & x_{n-1} & x_{n+2} & x_{n+3} & \cdots & x_{e-1} & x_0^2
\end{array}
\right].
\]

Then a minimal generating set $G_{S(e,m,n)}$ of $I_{S(e,m,n)}$ is given by $G_{S(e,m,n)} = J_{m,n} \cup K_{m,n}$, where $J_{m,n}$ denote the set of all $2 \times 2$ minors of $A_{m,n}$, and $K_{m,n}$ is given as follows.
}
{\rm 
\begin{enumerate}

\item $K_{2,3} = \ \begin{array}{l} \{ x_{n+1}x_{n+1+j}-x_{n-2}x_{n+4+j} \mid j \in [0, e-n-5]\}   \ \ \cup\\
 \{x_0^2x_{m-1}-x_{m+3}x_{e-4}, x_0x_1x_{m-1}-x_{m+3}x_{e-3}, x_1^3-x_4x_{e-1}\} \end{array} $

    \item $K_{2,4} = \ \begin{array}{l}\{ x_{m+1}x_{m+1+j}-x_{m-1}x_{m+3+j} \mid j \in [0,e-m-4]\setminus \{1\}\}  \ \cup\  \\  \{ x_{n+1}x_{n+1+j}-x_{n-1}x_{n+3+j} \mid j \in [0,e-n-4]\} \ \cup \\ \left\{\begin{array}{c}  x_0^2x_{m-1}-x_{m+1}x_{e-2}, x_0x_1x_{m-1}-x_{m+1}x_{e-1}, x_0x_1x_{n-1}-x_{n+2}x_{e-2}, \\ x_0^2x_{n-1}-x_{n+2}x_{e-3}, x_1^3-x_6x_{e-3}, x_1^2x_{n-1}-x_{n+2}x_{e-1}   \end{array}\right\} 
   \end{array} $ 

    \item $K_{2,5}= \ \begin{array}{l} \{x_{m+1}^2-x_{m-2}x_{m+4}\} \cup \{x_{m+1}x_{m+1+j}-x_{m-1}x_{m+3+j} \mid j \in [1,e-m-4] \setminus \{2\}\} \ \cup \\
    \{x_{n+1}x_{n+1+j}-x_{n-1}x_{n+3+j} \mid j \in [0,e-n-4] \} \ \cup \{x_{n-1}^2 - x_{n+2}x_{n-4}\} \ \cup \\ \left\lbrace \begin{array}{c} x_0x_1x_{n-1}-x_{n+2}x_{e-2}, x_0^2x_{n-1}-x_{n+2}x_{e-3}, x_1^2x_{n-1}-x_{n+2}x_{e-1},x_0^2x_{m-1}-x_{m+2}x_{e-3}, \\ x_0x_1x_{m-1}-x_{m+2}x_{e-2}, x_1^3 - x_4x_{e-1} \end{array} \right\rbrace
    \end{array}$
    
    \item For $6 \leq n \leq e-4 $,\\
    $K_{2,n}= \ \begin{array}{l} \{x_{m+1}x_{n-2}-x_{0}x_{n+1}\} \ \cup \\ \{x_{m+1}x_{m+1+j}-x_{m-1}x_{m+3+j} \mid j \in [0,e-m-4] \setminus \{n-m-1,n-m-3\}\} \ \cup \\
    \{x_{n+1}x_{n+1+j}-x_{n-1}x_{n+3+j} \mid j \in [0,e-n-4] \} \ \cup \\ \{x_{n-1}x_{n-1-j} - x_{n+1}x_{n-3-j} \mid j \in [0,n-m-4]\}  \ \cup \ \{x_{n-1}x_{m+2} - x_{n+2}x_{m-1}\} \ \cup \\ 
     \{x_0x_1x_{n-1}-x_{n+2}x_{e-2}, x_0^2x_{n-1}-x_{n+2}x_{e-3},x_0^2x_{m-1}-x_{m+2}x_{e-3}, x_0x_1x_{m-1}-x_{m+2}x_{e-2}, x_1^3 - x_4x_{e-1} \}
    \end{array}$

\item $K_{2,e-3} = \ \begin{array}{l} \{x_{m+1}x_{n-2}-x_{0}x_{n+1}\} \cup \{x_{m+1}x_{m+1+j}-x_{m-1}x_{m+3+j} \mid j \in [0,e-m-5] \setminus \{n-m-3\}\} \cup \\
     \{x_{n-1}x_{n-1-j} - x_{n+1}x_{n-3-j} \mid j \in [0,n-m-4]\}  \cup  \{x_{n-1}x_{m+2} - x_{n+2}x_{m-1}\} \ \cup \\  \{ x_0x_1x_{n-1}-x_{n+1}x_{e-1}, x_0^2x_{n-1}-x_{n+1}x_{e-2},x_0^2x_{m-1}-x_{m+3}x_{e-4}, x_0x_1x_{m-1}-x_{m+2}x_{e-2}, x_1^3 - x_4x_{e-1} \}\end{array}$

\item $K_{2,e-2} = \ \begin{array}{l} \{x_{m+1}x_{n-2}-x_{0}x_{n+1}\} \cup \{x_{m+1}x_{m+1+j}-x_{m-1}x_{m+3+j} \mid j \in [0,e-m-4] \setminus \{n-m-3\}\} \cup \\
     \{x_{n-1}x_{n-1-j} - x_{n+1}x_{n-3-j} \mid j \in [0,n-m-4]\}  \  \cup\\ \left\lbrace \begin{array}{c} x_0x_1x_{n-1}-x_{n+1}x_{e-1}, x_0^2x_{m-1}-x_{m+2}x_{e-3}, x_0x_1x_{m-1}-x_{m+3}x_{e-3}, x_1^3 - x_4x_{e-1}\end{array} \right\rbrace  \end{array}$

 \item $K_{3,4} = \ \begin{array}{l} \{ x_{n+1}x_{n+1+j}-x_{n-2}x_{n+4+j} \mid j \in [0, e-n-5]\}  \ \cup \\ \left\lbrace \begin{array}{c} x_0^2x_{m-1}-x_{m+3}x_{e-4}, x_0x_1x_{m-1}-x_{m+3}x_{e-3}, x_0x_2x_{n-2}-x_{n+2}x_{e-2}, x_2^2x_1-x_6x_{e-1}, x_2^3-x_7x_{e-1} \end{array} \right\rbrace \end{array}$

\item For $4 \leq m \leq e-6$,\\
$K_{m,m+1} = \ \begin{array}{l} \{ x_{n+1}x_{n+1+j}-x_{n-2}x_{n+4+j} \mid j \in [0, e-n-5]\}  \ \cup \\ 
\{x_{m-1}x_{m-1-j} - x_{m+2}x_{m-4-j} \mid j \in [0,m-4]\} \ \cup \\ \left\lbrace \begin{array}{c} x_0^2x_{m-1}-x_{m+2}x_{e-3}, x_0x_1x_{m-1} - x_{m+2}x_{e-2}, x_0x_2x_{n-2} - x_{n+2}x_{e-2}\end{array} \right\rbrace \end{array}$

\item $K_{e-5,e-4} = \ \begin{array}{l} \{x_{m-1}x_{m-1-j} - x_{m+2}x_{m-4-j} \mid j \in [0,m-4]\} \ \cup \\ \{x_0^2x_{m-1}-x_{m+2}x_{e-3}, x_0x_1x_{m-1} - x_{m+2}x_{e-2}, x_0x_2x_{n-2} - x_{n+2}x_{e-2}\}\end{array}$

\item $K_{e-4,e-3} = \ \begin{array}{l} \{x_{m-1}x_{m-1-j} - x_{m+2}x_{m-4-j} \mid j \in [0,m-4]\} \cup \{x_0x_1x_{m-1} - x_{m+2}x_{e-2}, x_0x_2x_{n-2} - x_{n+2}x_{e-2}\}\end{array}$

\item $K_{e-3,e-2} = \ \begin{array}{l} \{x_{m-1}x_{m-1-j} - x_{m+2}x_{m-4-j} \mid j \in [0,m-4]\} \cup \{x_0x_2x_{n-2} - x_{n+1}x_{e-1}\}\end{array}$

\item For $3 \leq m \leq e-6$,\\
$K_{m,m+2} = \ \begin{array}{l} \{x_{m+1}x_{m+1+j}-x_{m-1}x_{m+3+j} \mid j \in [0,e-m-4] \setminus \{1\}\} \ \cup \\ \{x_{n+1}x_{n+1+j} - x_{n-1}x_{n+3+j} \mid j \in [0,e-n-4]\} \cup \\
\{x_{n-1}x_{n-3-j} - x_{n+1}x_{n-5-j} \mid j \in [0,n-5]\} \ \cup \{x_{m-1}x_{m-1-j} - x_{m+1}x_{m-3-j} \mid j \in [0,m-3]\} \ \cup
\\ \{x_0^2x_{m-1} - x_{m+1}x_{e-2}, x_0x_1x_{m-1}-x_{m+1}x_{e-1}, x_0x_1x_{n-1} - x_{n+2}x_{e-2}, x_0^2x_{n-1} - x_{n+2}x_{e-3}\}\end{array}$    

\item $K_{e-5,e-3} = \ \begin{array}{l} \{x_{m+1}^2-x_{m-1}x_{m+3}\} \ \cup
\{x_{n-1}x_{n-3-j} - x_{n+1}x_{n-5-j} \mid j \in [0,n-5]\}  \cup \\ \{x_{m-1}x_{m-1-j} - x_{m+1}x_{m-3-j} \mid j \in [0,m-3]\} \ \cup \\  \{x_0^2x_{m-1} - x_{m+1}x_{e-2}, x_0x_1x_{m-1}-x_{m+1}x_{e-1}, x_0x_1x_{n-1} - x_{n+2}x_{e-2}, x_0^2x_{n-1} - x_{n+1}x_{e-2}\}\end{array}$

\item $K_{e-4,e-2} = \ \begin{array}{l} \{x_{m+1}^2-x_{m-1}x_{m+3}\} \ \cup
\{x_{n-1}x_{n-3-j} - x_{n+1}x_{n-5-j} \mid j \in [0,n-5]\}  \cup \\ \{x_{m-1}x_{m-1-j} - x_{m+1}x_{m-3-j} \mid j \in [0,m-3]\} \ \cup \  \{x_0x_1x_{m-1}-x_{m+1}x_{e-1}, x_0x_1x_{n-1} - x_{n+1}x_{e-1}\}\end{array}$   

\item For $m \geq 3, m+2 < n \leq e-4$, \\ $K_{m,n} = \ \begin{array}{l} \{x_{m+1}x_{n-m-5}-x_{m-2}x_{n-m-2}\} \ \cup \\ \{x_{m+1}x_{m+1+j}-x_{m-1}x_{m+3+j} \mid j \in [0,e-m-4] \setminus \{n-m-1,n-m-3\}\} \cup \\ \{x_{n+1}x_{n+1+j} - x_{n-1}x_{n+3+j} \mid j \in [0,e-n-4]\} \ \cup  \\
\{x_{n-1}x_{n-1-j} - x_{n+1}x_{n-3-j} \mid j \in [0,n-3] \setminus \{n-m-1,n-m-2,n-m-3\}\} \cup \\ \{x_{n-1}x_{m+2} - x_{n+2}x_{m-1}\} \cup  \{x_{m-1}x_{m-1-j} - x_{m+1}x_{m-3-j} \mid j \in [0,m-3]\} \ \cup
\\ \left\lbrace \begin{array}{c} x_0^2x_{m-1}-x_{m+2}x_{e-3}, x_0x_1x_{m-1}-x_{m+2}x_{e-2},x_0x_1x_{n-1}-x_{n+2}x_{e-2}, x_0^2x_{n-1}-x_{n+2}x_{e-3} \end{array} \right\rbrace \end{array}$

\item For $m \geq 3, m+2 < n = e-3$,\\
$K_{m,n} = \ \begin{array}{l} \{x_{m+1}x_{n-m-5}-x_{m-2}x_{n-m-2}\} \ \cup \\ \{x_{m+1}x_{m+1+j}-x_{m-1}x_{m+3+j} \mid j \in [0,e-m-4] \setminus \{n-m-1,n-m-3\}\} \cup  \\
\{x_{n-1}x_{n-1-j} - x_{n+1}x_{n-3-j} \mid j \in [0,n-3] \setminus \{n-m-1,n-m-2,n-m-3\}\} \cup \\ \{x_{n-1}x_{m+2} - x_{n+2}x_{m-1}\} \cup  \{x_{m-1}x_{m-1-j} - x_{m+1}x_{m-3-j} \mid j \in [0,m-3]\} \ \cup
\\ \left\lbrace \begin{array}{c} x_0^2x_{m-1} - x_{m+1}x_{e-2}, x_0x_1x_{m-1}-x_{m+1}x_{e-1}, x_0x_1x_{n-1} - x_{n+2}x_{e-2}, x_0^2x_{n-1} - x_{n+1}x_{e-2} \end{array} \right\rbrace \end{array} $

\item For $m \geq 3, m+2 < n = e-2$,\\
$K_{m,n} = \ \begin{array}{l} \{x_{m+1}x_{n-m-5}-x_{m-2}x_{n-m-2}\} \ \cup \\ \{x_{m+1}x_{m+1+j}-x_{m-1}x_{m+3+j} \mid j \in [0,e-m-4] \setminus \{n-m-3\}\} \cup  \\
\{x_{n-1}x_{n-1-j} - x_{n+1}x_{n-3-j} \mid j \in [0,n-3] \setminus \{n-m-1,n-m-2,n-m-3\}\} \cup \\  \{x_{m-1}x_{m-1-j} - x_{m+1}x_{m-3-j} \mid j \in [0,m-3]\} \ \cup
\\ \left\lbrace \begin{array}{c} x_0^2x_{m-1}-x_{m+2}x_{e-3}, x_0x_1x_{m-1}-x_{m+1}x_{e-1}, x_0x_1x_{n-1} - x_{n+1}x_{e-1} \end{array} \right\rbrace\end{array}$
\end{enumerate}
}\end{theorem}

We will show that $\mathbb K[\overline{S(e,m,n)}]$ is Cohen--Macaulay if and only if $(m,n)\neq (e-4,e-3)$ (\cref{cor:CM-characterization}). We first prove the ``easy" direction.

\begin{lemma}\label{lem:notCM}
    The ring $\mathbb K[\overline{S(e,e-4,e-3)}]$ is not Cohen--Macaulay.
\end{lemma}

\begin{proof}
A direct computation using \texttt{Macaulay2}~\cite{M2} shows that the result holds for $e \leq 9$. Thus, we assume that $e\geq 10$. We observe that $x_0^2x_{e-5}x_{e-1}-x_{e-2}^3 \in I_{S(e,e-4,e-3)}$, since $\wt(x_0^2x_{e-5}x_{e-1})=\wt(x_{e-2}^3)$.    In particular, the monomial $x_0^2x_{e-5}x_{e-1}$ is in $\operatorname{in}_< (I(e,e-4,e-3))$. Since this monomial is divisible by $x_{e-1}$, we are done with the proof by \Cref{thm:CM-criterion} if we show that $x_0^2x_{e-5}$ is not in $\operatorname{in}_< (I(e,e-4,e-3))$.

     Suppose the sake of contradiction that \begin{equation}\label{eq1}
         x_0^2x_{e-5}\in \operatorname{in}_< (I(e,e-4,e-3)).
     \end{equation}
     Since $\mathcal G_{S(e,e-4,e-3)}$ consists of binomials, applying the Buchberger's algorithm to this set we obtain a Gr\"obner basis of $I_{S(e,e-4,e-3)}$ consisting of only binomials. Due to (\ref{eq1}), the ideal $I(e,e-4,e-3)$ contains a binomial $x_0^2x_{e-5}-M$ where $M<x_0^2x_{e-5}$ is a monomial with $\wt(M)=\wt(x_0^2x_{e-5})=4e-5$.  As $M<x_0^2x_{e-5}$, we have $\deg (M) \leq 3$. We have the following claim.
     
     \begin{claim}\label{clm}
         We have that $\gcd(M,x_0^2x_{e-5})=1$. 
     \end{claim}
     \begin{proof}[Proof of \cref{clm}]
         Suppose that $\gcd(M,x_0^2x_{e-5})\neq 1$. Then by taking out a common variable of the two, we obtain two monomials $M''<M'$ with the same weight, where $M'\in \{x_0^2,x_0x_{e-5}\}$. Since $M''<M'$, we have $\deg(M'')\leq \deg(M')=2$. It is straightforward that $\wt(U)<2e$ for any monomial $U$ of degree $1$, and $\wt(V)>2e$ for any monomial $V\neq x_0^2$ of degree $2$. Therefore, if $M'=x_0^2$, we obtain a contradiction as $\wt(x_0^2)=2e$. Now assume that $M'=x_0x_{e-5}$, and in particular, $\wt(M')=4e-5$. Then first we must have $\deg(M'')=2$. Set $M''=x_ix_j$ for some $i<j$. If $i=0$, then $\wt(M'')=\wt(M')$ implies that $j=e-5$, or equivalently, $M''=M'$, a contradiction. Now assume that $i\geq 1$. Then $j\leq e-6$. In this case, $M''>M'$, a contradiction.
     \end{proof}

     Since $\wt(M)=4e-5$, the monomial $M$ cannot be linear. Next suppose that $\deg(M)=2$, i.e., $M=x_ix_j$ for $i,j\in [0,e-1]\setminus \{e-4,e-3\}$ where $i<j$. Then it is straightforward that $\wt(M)\neq 4e-5$, a contradiction. Now we can assume that $\deg M=3$, i.e., $M=x_ix_jx_k$ for $i,j,k\in [0,e-1]$ where $i<j<k$. By \cref{clm}, we have $i,j\geq 1$, and since $M<x_0^2x_{e-5}$, we have $k> e-5$. Therefore, we have
     \[
     \wt(M) = 3e+i+j+k\geq 4e-3>4e-5,
     \]
     a contradiction, as desired.
\end{proof}

To prove that $\mathbb K[\overline{S(e,m,n)}]$ is Cohen--Macaulay when $(m,n)\neq (e-4,e-3)$, we need a Gr\"obner basis of $I_{\overline{S(e,m,n)}}$ and then apply \cref{thm:CM-criterion}. Recall that a Gr\"obner basis of $I_{\overline{S(e,m,n)}}$ can be obtained from one of $I_{S(e,m,n)}$ by homongenization (\cite[Proposition~3.15]{Ene-Herzog}).

The generating set $G_{S(e,m,n)}$ given in \Cref{thm:gens-for-m=1} and \Cref{thm:gens-for-m-geq-2} is not always a Gr\"obner basis of $I_{S(e,m,n)}$. Thus, we need to add more elements to $G_{S(e,m,n)}$ to obtain a Gr\"obner basis of $I_{S(e,m,n)}$. For every $(e,m,n)$, we define a subset $G'_{S(e,m,n)}$ of $I_{S(e,m,n)}$ as in \Cref{t:table} below.

\begin{table}[h!]
\centering
\setlength{\tabcolsep}{6pt} 
\renewcommand{\arraystretch}{1.2} 
\begin{tabular}{|c|c|c|}
\hline
\( m \) & \( n \) & \( G'_{S(e,m,n)} \) \\
\hline
\multirow{4}{*}{$1$} 
  & $2, 3$, \( e-2 \) & \( \emptyset \) \\ \cline{2-3}
  & $4$              &  $\{ x_0x_5x_{e-2}-x_2^2x_{e-1}, x_2^2x_{e-2}-x_0x_3x_{e-1} \}$               \\ \cline{2-3}
  & \( [5, e-4] \) &    $\{ x_0x_{n+1}x_{e-2} - x_2x_{n-2}x_{e-1} \}$            \\ \cline{2-3}
  & \( e-3 \)      &    $\{x_{e-2}^3-x_{e-4}x_{e-1}^2, x_2x_{e-2}^2-x_0x_{e-1}^2, x_0x_{e-2}^2-x_2x_{e-5}x_{e-1} \}$              \\ 
\hline
\multirow{5}{*}{$2$} 
  &  $3$    &   $\{ x_0x_4x_{e-3}-x_1^2x_{e-1} \}$              \\ \cline{2-3}
  & $4, 5$ &  $\{ x_0x_3x_{e-2}-x_1^2x_{e-1}, x_0x_{n+1}x_{e-2}-x_1x_{n-1}x_{e-1}\}$  \\ \cline{2-3}
  & $[6, e-4]$   & $\left\lbrace \begin{array}{c} x_0x_{n+1}x_{e-2}-x_1x_{n-1}x_{e-1}, x_0x_3x_{e-2}-x_1^2x_{e-1}, \\ x_1^2x_{n-1}-x_{n+2}x_{e-1} \end{array} \right\rbrace$            \\ \cline{2-3}
  & $e-3$ & $\left\lbrace \begin{array}{c}
       x_{e-2}^3-x_{e-4}x_{e-1}^2, x_3x_{e-2}^2-x_1x_{e-1}^2, x_0x_{e-2}^2-x_1x_{e-4}x_{e-1}, \\
       x_0x_3x_{e-2}-x_1^2x_{e-1}, x_1^2x_{e-4}- x_{e-1}^2 
       \end{array} \right\rbrace$ \\ \cline{2-3}
  & $e-2$ & $\{ x_0x_4x_{e-3}-x_1^2x_{e-1}, x_1^2x_{e-3}-x_0^2x_{e-1} \}$\\ \cline{2-3}
\hline
\multirow{4}{*}{$[3 ,e-6]$}
      &  $n < e-3$ and $n=m+1$   &   $\{x_0x_{m+2}x_{e-3}-x_1x_{m-1}x_{e-1}, x_0x_{n+1}x_{e-2} - x_2x_{n-2}x_{e-1} \}$ \\ \cline{2-3}
     & $n< e-3$ and $n\neq m+1$ & $\{x_0x_{m+1}x_{e-2}-x_1x_{m-1}x_{e-1}, x_0x_{n+1}x_{e-2} - x_1x_{n-1}x_{e-1}  \}$ \\ \cline{2-3}
     & $ e-3$ & $ \left\lbrace \begin{array}{c} x_{e-2}^3-x_{e-4}x_{e-1}^2, x_{m+1}x_{e-2}^2-x_{m-1}x_{e-1}^2, \\x_0x_{e-2}^2-x_1x_{e-4}x_{e-1}, x_0x_{m+1}x_{e-2} - x_1x_{m-1}x_{e-1} \end{array} \right\rbrace$ \\ \cline{2-3}
      & $e-2$ & $\{x_0x_{m+2}x_{e-3}-x_1x_{m-1}x_{e-1} \}$ \\ \cline{2-3}   
    \hline
    \multirow{3}{*}{$e-5$}
      &  $e-4$    &   $\left\lbrace \begin{array}{c} x_{e-3}^2x_{e-2}-x_{e-6}x_{e-1}^2, x_0x_{e-3}x_{e-2} - x_2x_{e-6}x_{e-1},\\ x_{e-3}^3 - x_{e-7}x_{e-1}^2, x_0x_{e-3}^2 - x_1x_{e-6}x_{e-1} \end{array} \right\rbrace$ \\ \cline{2-3}
     & $ e-3$ & $\left\lbrace \begin{array}{c} x_{e-2}^3-x_{e-4}x_{e-1}^2, x_{e-4}x_{e-2}^2 - x_{e-6}x_{e-1}^2,\\
     x_0x_{e-2}^2 - x_1x_{e-4}x_{e-1}, x_0x_{e-4}x_{e-2} - x_1x_{e-6}x_{e-1} \end{array} \right\rbrace$ \\ \cline{2-3}
      & $e-2$ & $\{x_{e-3}^3 - x_{e-7}x_{e-1}^2, x_1x_{e-3}^2 - x_0x_{e-4}x_{e-1}, x_0x_{e-3}^2 - x_1x_{e-6}x_{e-1} \}$ \\ \cline{2-3}
    \hline
     \multirow{2}{*}{$e-4$}
      &  $e-3$    &   $\{x_0x_{e-2}^2 - x_2x_{e-5}x_{e-1}, x_0^2x_{e-5}x_{e-1}-x_{e-2}^3, x_{e-2}^4-x_{e-5}x_{e-1}^3 \}$ \\ \cline{2-3}
     & $ e-2$ & $ \emptyset$ \\ \cline{2-3}
    \hline
    \multirow{1}{*}{$e-3$}
     & $ e-2$ & $\emptyset$ \\ \cline{2-3}
    \hline
\end{tabular}
\vspace{.2cm}
\caption{The sets $G'_{S(e,m,n)}$.}
\label{t:table}
\end{table}

Define $\mathcal{G}_{S(e,m,n)}= {G}_{S(e,m,n)} \cup {G}'_{S(e,m,n)}$.  We will show that $\mathcal{G}_{S(e,m,n)}$ is indeed a Gr\"obner basis for every $e \geq 10$ and $(m,n)\neq (e-4, e-3)$. Let $\operatorname{in}_<\left(\mathcal{G}_{S(e,m,n)}\right)$ denote the ideal of $R(e,m,n)$ generated by the initial terms of elements of $\mathcal{G}_{S(e,m,n)}$. We use the same notation to denote its image in $R'(e,m,n)$. 

The following two lemmas list the monomials in $R'(e,m,n)$ which do not belong to $\operatorname{in}_<(\mathcal{G}_{S(e,m,n)})$. 

\begin{lemma}\label{rem:Surviving-monomials-in-degrees-2-and-3}
    Let $e\geq 10$, $(m,n)\neq (e-4, e-3)$. Consider $\operatorname{in}_<(\mathcal{G}_{S(e,m,n)})$ as an ideal of $R'(e,m,n)$. Then

    \begin{enumerate}
        \item For $m=1$ and $0 \leq i\leq j \leq e-2$, $x_ix_j \not\in \operatorname{in}_<(\mathcal{G}_{S(e,m,n)})$ if and only if $i=0$ or 
        \[   (i,j) = \left\{
\begin{array}{ll}
(2,e-2) & ~\text{for all }~ n\\
(2,n-2) & ~\text{for all }~ n\\
(n+1,e-3) & ~\text{if}~ n=2\\
(n+1,e-2) & ~\text{for all}~  n\\
(3, e-3)& ~\text{if}~ n=e-2 .\end{array} 
\right. \]
        \item For $m \geq 2$ and $0 \leq i \leq j \leq e-2$, $x_ix_j \not\in \operatorname{in}_<(\mathcal{G}_{S(e,m,n)})$ if and only if $i=0$ or 
        \[   (i,j) = \left\{
\begin{array}{ll}
(m+1,e-2) & ~\text{for all }~ (m,n)\\
(n+1,e-3) & ~\text{if}~ n=m+1 \\
(n+1,e-2) & ~\text{for all}~ (m,n) \\
(1,m-1) & ~\text{for all}~ (m,n) \\
       (2,m-1) & ~\text{if}~ n=m+1  \\
       (1,n-1) & ~\text{for all}~(m,n)  \\
              (m+2,n-1) & ~\text{if}~ n=e-2  .\\

\end{array} 
\right. \]

        \item For $0 \leq i, j \leq e-2$ with $i \neq j$, we have $x_i^2 x_j \not\in \operatorname{in}_<(\mathcal{G}_{S(e,m,n)})$ if and only if $i=0$ and 
\[   j = \left\{
\begin{array}{ll}
e-3 & ~\text{if}~ (m,n) = (1,e-2) \\
       e-3 & ~\text{if}~(m,n) = (2,e-2) \\
       e-4 & ~\text{if}~(m,n) = (e-3,e-2) \\
       e-5 \text{ \ or\  }e-3 & ~\text{if}~(m,n) = (e-4,e-2) \\
       e-3 & ~\text{if}~(m,n) = (m,e-2); 3 \leq m \leq e-5 .\\
\end{array} 
\right. \]

        \item For $0 \leq i < j < k \leq e-2$, we have $x_ix_j x_k \not\in \operatorname{in}_<(\mathcal{G}_{S(e,m,n)})$ if and only if $i=0$ and
        \[   (j,k) = \left\{
\begin{array}{ll}
(2,e-2) & ~\text{if}~ (m,n) = (1,n); 3 \leq n \leq e-3\\
(3,e-3) & ~\text{if}~ (m,n) = (1,2),~ (1,e-2) \\
(3,e-2) & ~\text{if}~ (m,n) = (1,2) \\
(4,e-2) & ~\text{if}~ (m,n) = (1,3)\\
       (4,e-2) & ~\text{if}~ (m,n) = (2,3)  \\
       (1,e-4) & ~\text{if}~(m,n) = (e-3,e-2). \\
\end{array} 
\right. \]
    \end{enumerate}
\end{lemma}

\begin{proof}
    This is straightforward from the definition of $\mathcal{G}_{S(e,m,n)}$.
\end{proof}

\begin{lemma}\label{lem:m^4-general}
    Let $e \geq 10$, $(m,n) \neq (e-4, e-3)$, and $\mathfrak{m}$ be the homogeneous maximal ideal of $\dfrac{R(e,m,n)}{\operatorname{in}_<(\mathcal{G}_{S(e,m,n)})+(x_{e-1})}$. Then $\mathfrak{m}^4 = 0$.
\end{lemma}
\begin{proof}
Note that $\mathfrak{m}^4$ is generated by the elements of the following type: (a) $x_i^4$, (b) $x_i^3x_j$, (c) $x_i^2x_j^2$, (d) $x_i^2x_jx_k$, (e) $x_ix_jx_kx_l$; where $i,j,k,l \leq e-2$ and are all distinct.

By \Cref{rem:Surviving-monomials-in-degrees-2-and-3}, we see that elements of type (a) and (b) are in $\operatorname{in}_<(\mathcal{G}_{S(e,m,n)})$. 

For (c),  we need to only consider the elements of the form $x_0^2x_{j}^2$ for the cases mentioned in \Cref{rem:Surviving-monomials-in-degrees-2-and-3}(3). Among these cases, from \Cref{rem:Surviving-monomials-in-degrees-2-and-3}(1), (2), we see that $x_{j}^2 \in \operatorname{in}_<(\mathcal{G}_{S(e,m,n)})$ in every case except for $j=e-3$ and $(m,n)=(e-5,e-2)$. But in this case as well, from \Cref{t:table}, we see that $x_0x_{e-3}^2 \in \operatorname{in}_<(\mathcal{G}_{S(e,m,n)})$. 

For (d), from \Cref{rem:Surviving-monomials-in-degrees-2-and-3}, we only need to check for the case $(i,j,k)=(0,e-5,e-3)$ with $(m,n)=(e-4, e-2)$. But in this case as well, we see that $x_{e-5}x_{e-3}\in \operatorname{in}_<(\mathcal{G}_{S(e,m,n)})$, since $x_{e-5}x_{e-3}-x_{e-7}x_{e-1} \in \mathcal{G}_{S(e,m,n)}$.

For (e), from \Cref{rem:Surviving-monomials-in-degrees-2-and-3}(4) we first see that $i=0$, and in addition, the only elements that require checking are as follows: (i) $x_0x_3x_{e-3}e_{e-2}$ for $(m,n)=(1,2)$, (ii) $x_0x_2x_4x_{e-2}$ for $(m,n)=(1,3)$. In case (i), $x_{e-3}x_{e-2} \in \operatorname{in}_<(\mathcal{G}_{S(e,m,n)})$ from \Cref{rem:Surviving-monomials-in-degrees-2-and-3}(1). In case (ii), $x_2x_4 \in \operatorname{in}_<(\mathcal{G}_{S(e,m,n)})$ from \Cref{rem:Surviving-monomials-in-degrees-2-and-3}(1).
\end{proof}

We are now ready to prove that $\mathcal{G}_{S(e,,m,n)}$ is a Gr\"obner basis of $I_{S(e,m,n)}$ under appropriate hypotheses.

\begin{theorem}\label{thm:GB-proof-general}
    Let $e \geq 10$ and $(m,n)\neq (e-4, n-3)$. Then the set $\mathcal{G}_{S(e,,m,n)}$ is a Gr\"obner basis of $I_{S(e,m,n)}$. 
\end{theorem}
\begin{proof}
Let $e\geq 10$ and $(m,n) \neq (e-4, e-3)$. Suppose that $R(e,m,n)$ is equipped with graded reverse lexicographic order $<$. Fix an $(e,m,n)$, and let $\mathcal{G}$ denote the set $\mathcal{G}_{S(e,m,n)}=G_{S(e,m,n)}\cup G'_{S(e,m,n)}$. Suppose $J$ is the ideal generated by $\mathcal{G}$, and $I=I_{S(e,m,n)}$. We want to show that $\mathcal{G}$ is a Gr\"obner basis of $I$. Note that since $\mathcal{G}$ contains a generating set of $I$, $\mathcal{G}$ itself is a generating set of $I$.

Let $\widetilde{\mathcal{G}}$ denote the Gr\"obner basis of $I$ obtained by applying Buchberger's algotithm to $\mathcal{G}$. Observe that if $f$ and $g$ are binomials, then their S-pair is also a binomial. Thus, it follows that $\widetilde{\mathcal{G}}$ consists of binomials. Furthermore, if there is an $x_i$ which divides an element $f\in \widetilde{\mathcal{G}}$, then $f/x_i \in I$, and so we may replace $\widetilde{\mathcal{G}}$ with $\left(\widetilde{\mathcal{G}}\setminus \{f\}\right)\ \cup \{f/x_i\}$, which is still a Gr\"obner basis of $I$. Thus,  without loss of generallity, we may assume that no element in $\widetilde{\mathcal{G}}$ is divisible by a variable. Observe that no element of $\mathcal{G}$ is divisible by a variable. So, we have $\mathcal{G} \subseteq \widetilde{\mathcal{G}}$. Moreover, we see that none of the initial terms of elements of $\mathcal{G}$ is divisible by $x_{e-1}$.

We first claim that $\widetilde{\mathcal{G}}$ can be modified in such a way that $x_{e-1}$ does not divide the initial term of any of its elements. We prove the claim by induction on the $\ell$, the number of elements of $\widetilde{\mathcal{G}}$ whose initial term is divisible by $x_{e-1}$. The base case, i.e., the case $\ell=0$ is clear. Assume that the result holds for some integer $\ell-1\geq 0$. We now prove the result for $\ell$. So, assume that $\ell\geq 1$, and that there are $\ell$ elements in $\widetilde{\mathcal{G}}$ whose initial terms are divisible by $x_{e-1}$. Suppose that one such term is $f =M_1-N_1 \in \widetilde{\mathcal{G}}$, with $\operatorname{in}_<(f)=M_1$ and $x_{e-1}\mid M_1$. Then $x_{e-1}\nmid N_1$. Observe that since $x_{e-1}$ is the least variable with respect to $<$, we must have $\deg(M_1)>\deg(N_1)$. If $N_1$ is divisible by any of the initial terms of $\mathcal{G}$, say $N_1=M\cdot M_2=M\cdot \operatorname{in}_<(M_2-N'_2)$ for some $M_2-N_2' \in \mathcal{G}$, then we see that with $N_2=M\cdot N_2'$, the element $M_1-N_2 \in I$ with $\operatorname{in}_<(M_1-N_2)=M_1$. Note that with respect to $<$, we have $N_1>N_2$. If $N_2$ is divisible by the initial term of some element of $G$, then proceeding as above, we get an element $M_1-N_3 \in I$ with $\operatorname{in}_<(M_1-N_3)=M_1$ and $N_1>N_2>N_3$. Since every strictly descending chain of monomials with respect to $<$ starting with $N_1$ is finite, we see that after a finite number of iterations of the above process we would obtain an element $N_r$ such that $N_1>N_2>\cdots>N_r$, $M_1-N_r \in I$, $\operatorname{in}_<(M_1-N_r)=M_1$, and $N_r$ is not divisible by any of the initial terms of elements of $\mathcal{G}$. Thus, we have two cases: \begin{itemize}
    \item[(i)] $x_{e-1} \mid M_1-N_r$,
    \item[(ii)] $x_{e-1} \nmid M_1-N_r$. 
\end{itemize}
In case (i), we may replace $\widetilde{\mathcal{G}}$ by $\left(\widetilde{\mathcal{G}} \setminus \{M_1-N_1\}\right)  \cup \left\{(M_1-N_r)/\gcd(M_1, N_r)\right\}$. Observe that the latter set contains $G$, is again a Gr\"obner basis of $I$, and none of its elements is divisible by any variable. Now, if $x_{e-1} \nmid M_1/\gcd(M_1, N_r)$, then the new Gr\"obner basis has exactly $\ell-1$ elements whose initial term is divisible by $x_{e-1}$, and we are done by induction. On the other hand, if $x_{e-1} \mid M_1/\gcd(M_1, N_r)$, then $x_{e-1} \nmid N_r/\gcd(M_1, N_r)$, and we are in case (ii). \\
To complete the proof of the claim, we now show that case (ii) does not occur at all. For the sake of contradiction, suppose that $x_{e-1}\nmid M_1-N_r \in \widetilde{\mathcal{G}}$, $x_{e-1}\mid M_1$, and $\gcd(M_1, N_r)=1$. Then $N_r$ is not divisible by $x_{e-1}$. Note that by construction, $N_r$ is not divisible by any of the initial terms of $\mathcal{G}$. Moreover, since $x_{e-1}\mid M_1$ and $x_{e-1}\nmid N_r$, we have $\deg(M_1)>\deg(N_r)$. \\
Observe that from  \Cref{lem:m^4-general}, we must have $\deg(N_r)\leq 3$. Since $M_1-N_r\in I$, we also have $\deg(N_r)\geq 2$. Thus, $\deg(N_r)=2$ or $\deg(N_r)=3$. Assume that $\deg(N_r)=2$. In this case, $\deg(M_1)\geq 3$. Moreover, since $x_{e-1}\mid M_1$, we have $\wt(M_1)\geq e+e+(2e-1)=4e-1$. On the other hand, since $x_{e-1}\nmid N_r$, we have $\wt(N_r)\leq (2e-2)+(2e-2)=4e-4< \wt(M_1)$. This contradicts the fact that $M_1-N_r\in I$. Therefore, we may assume that $\deg(N_r)=3$. In this case, $\deg(M_1)\geq 4$. Thus, we have $\wt(M_1)\geq e+e+e+(2e-1)=5e-1$. Recall that \Cref{rem:Surviving-monomials-in-degrees-2-and-3} precisely gives us the set of all possibilities for $N_r$. If $N_r$ is of the form $x_i^2x_j$, then \Cref{rem:Surviving-monomials-in-degrees-2-and-3}(3) tells us that $i=0$. In particular, $\wt(N_r)\leq e+e+(2e-2)=4e-2<5e-1\leq \wt(M_1)$, which is again a contradiction. Suppose $N_r$ is of the form $x_ix_jx_k$. Then \Cref{rem:Surviving-monomials-in-degrees-2-and-3}(4) tells us that $i=0$, $j\leq 4$ and $k\leq e-2$, which again gives us the contradiction that $\wt(N_r)\leq e+(e+4)+(2e-2)=4e+2<5e-1\leq \wt(M_1)$. 
\\
Thus, by induction, we have proved that $\widetilde{\mathcal{G}}$ can be modified such that $x_{e-1}$ does not divide initial term of any of its elements. 

Now, in view of the claim proved above, assume that $\widetilde{\mathcal{G}}$ is a binomial Gr\"obner basis of $I$ containing $\mathcal{G}$  such that $x_{e-1}$ does not divide the initial term of any element of $\widetilde{\mathcal{G}}$. Furthermore, we may assume that $\widetilde{\mathcal{G}}$ is minimal in the sense that given any $f, g \in \widetilde{\mathcal{G}}$, neither  the monomials $\operatorname{in}_<(f) $ nor $ \operatorname{in}_<(g)$ divides the other. We prove that $\widetilde{\mathcal{G}}=\mathcal{G}$. For this, from \Cref{lem:m^4-general}, it is enough to show that if $M$ is a monomial of degree $2$ or $3$, which is not a multiple of any of the initial term of an element of $\mathcal{G}$, then $M$ is not an initial term of any binomial in $I$. We do this by the method of contradiction. So, assume the existence of an $M$ as above so that $f=M-N \in \widetilde{\mathcal{G}}$.

\textit{Case 1:} $\deg(M)=2$.\\
In this case, we must have $\deg(N)=2$. Hence, if $M=x_{i}x_j$, then $N$ has the form $x_{i-k}x_{j+k}$ for some $k>0$. Now, in every case listed in \Cref{rem:Surviving-monomials-in-degrees-2-and-3}(1),(2), it is easy to see that given any $x_ix_j$ and $k>0$, either $x_{i-k}$ or $x_{j+k}$ is not a variable in $R(e,m,n)$. For instance, consider the case $(i,j)=(3,e-3)$ and $(m,n)=(1,e-2)$ from \Cref{rem:Surviving-monomials-in-degrees-2-and-3}(1). Then, we see that neither of the monomials $x_2x_{e-2}, x_{1}x_{e-1}$ is in $R(e,,m,n)$ because of the absence of $x_1$ and $x_{e-2}$. 
Thus, we are done in Case 1.

\textit{Case 2:} $\deg(M)=3$.\\
By Case (1) above, we may assume that $\gcd(M,N)=1$. Here, $N$ can have degree 2 or 3. Let $M=x_ix_jx_k$ with $i\leq j\leq k$, $N=x_{i'}x_{j'}$ if $\deg(N)=2$, and $N=x_{i'}x_{j'}x_{k'}$ if $\deg(N)=3$, where $i'\leq j'\leq k'$. Observe that when $\deg(N)=3$, we must have $k<k'$. If $\deg(N)=2$, then $\wt(N)\leq 4e-2$. 

We first analyze all cases listed in \Cref{rem:Surviving-monomials-in-degrees-2-and-3}(4), i.e., when $M$ is of the form $x_ix_jx_k$ with $i<j<k$. \\
For the first five cases listed therein, we see that $\wt(M)\geq 4e$. Thus, $\deg(N)$ must be equal to $3$. In these five cases, it is easy to arrive at the contradiction $\wt(N)>\wt(M)$ using the fact that $\gcd(M, N)=1$. For instance, consider the case $M=x_0x_3x_{e-3}$ and $(m,n)=(1,e-2)$. Then, we must have $i'\geq 2, j'\geq 3, k'=e-1$, which forces $\wt(N)\geq (e+2)+(e+3)+(2e-1)> 4e=\wt(M)$.\\
For the remaining case $(i,j,k)=(0,1,e-4)$ and $(m,n)=(e-3,e-2)$, we have $\wt(M)=4e-3$. The only way $N=x_{i'}x_{j'}$ can have weight $4e-3$ is if $N=x_{e-2}x_{e-1}$, which is impossible. Thus, $\deg(N)=3$. But now, $i'\geq 2, j'\geq 3, k'=e-1$ gives $\wt(N)>\wt(M)$, a contradiction.

Finally, we analyze all cases listed in \Cref{rem:Surviving-monomials-in-degrees-2-and-3}(3), i.e., when $M$ is of the form $x_ix_jx_k$ with $i=j<k$. \\
In all these cases, we have $\wt(M)\leq 4e-3$, and in addition, if $N=x_{i'}x_{j'}x_{k'}$, then $k<k'$ forces $k'=e-1$. As a consequence, the conditions $i'\geq 1, j'\geq 2, k'=e-1$ give us $\wt(N)\geq 4e+2>\wt(M)$. This implies $\deg(N)=2$. So, let $N=x_{i'}x_{j'}$. Then we have $j'=e-1$, and hence comparing $\wt(M)$ and$\wt(N)$, we see that for $j=e-3, e-4, e-5$,  the index $i'$ is forced to be equal to $e-2, e-3, e-4$, respectively. However, in each case, we see that $i'$ is either $m$ or $n$, which is impossible. 
  
The above analysis shows that $\mathcal{G}=\widetilde{\mathcal{G}}$, completing the proof.
\end{proof}

\begin{corollary}\label{cor:GB-for-projective}
    Let $e \geq 10$ and $(m,n)\neq (e-4, n-3)$. Suppose that $\mathcal{G}_{S(e,m,n)}^h$ denotes the homogenization of $\mathcal{G}_{S(e,m,n)}$. Then $\mathcal{G}_{S(e,m,n)}^h$ is a Gr\"obner basis of $I_{\overline{S(e,m,n)}}$. 
\end{corollary}
\begin{proof}
    This follows from \Cref{thm:GB-proof-general} and \cite[Proposition 3.15]{Ene-Herzog}.
\end{proof}
As a consequence of the proof of \Cref{thm:GB-proof-general}, we obtain the following characterization for Cohen--Macaulayness of $\mathbb K[\overline{S(e,m,n)}]$.
\begin{corollary}\label{cor:CM-characterization}
    The ring $\mathbb{K}[\overline{S(e,m,n)}]$ is Cohen--Macaulay if and only if $(m,n) \neq (e-4,e-3)$.
\end{corollary}
\begin{proof}

A direct computation using \texttt{Macaulay2}~\cite{M2} shows that that the result holds when $4 \leq e \leq 9$. Thus, we assume that $e \geq 10$. 

As noted earlier (\Cref{lem:notCM}), if $(m,n)=(e-4, e-3)$, then $\mathbb{K}[\overline{S(e,m,n)}]$ is not Cohen--Macaulay. If $(m,n)\neq (e-4, e-3)$, then by \Cref{thm:GB-proof-general}, it follows that $x_{e-1}$ does not divide the initial term of any element of the set $\mathcal{G}_{S(e,m,n)}$. Hence, by \Cref{thm:CM-criterion}, we conclude that $\mathbb{K}[\overline{S(e,m,n)}]$ is Cohen--Macaulay.
\end{proof}

Next we obtain a characterization for the property of being Gorenstein or complete intersection.



\begin{theorem}\label{thm:Gor}
The ring $\mathbb{K}[\overline{S(e,m,n)}]$ is Gorenstein if and only if it is a complete intersection if and only if $(e,m,n) = (4,1,2)$ or $(e,m,n) = (5,2,3)$. 
\end{theorem}
\begin{proof}
A direct computation using \texttt{Macaulay2}~\cite{M2} shows that for $e \leq 9$, the ring $\mathbb{K}[\overline{S(e,m,n)}]$ is Gorenstein if and only if $(e,m,n) = (4,1,2)$ or $(e,m,n) = (5,2,3)$. In fact, in both the cases, $\mathbb{K}[\overline{S(e,m,n)}]$ is a complete intersection. Now, suppose that $e \geq 10$. We show that $\mathbb{K}[\overline{S(e,m,n)}]$ is not Gorenstein. When $(m,n)=(e,4,e-3)$, from \Cref{cor:CM-characterization}, we have that $ \mathbb{K}[\overline{S(e,m,n)}]$ is not Cohen--Macaulay, and hence not Gorenstein. Thus, assume that $(m,n)\neq (e,4,e-3)$. In this case, we now have that $\mathbb{K}[\overline{S(e,m,n)}]$ is Cohen--Macaulay of dimension 2. 

Since $\{y, x_{e-1}\}$ is a system of parameters for $\mathbb{K}[\overline{S(e,m,n)}]$, it is a regular sequence. Hence, $\mathbb{K}[\overline{S(e,m,n)}]$  is Gorenstein if and only if $\dfrac{\mathbb{K}[\overline{S(e,m,n)}]}{(y, x_{e-1})}$ is Gorenstein Artinian ring. In particular, the Hilbert function of $\dfrac{\mathbb{K}[\overline{S(e,m,n)}]}{(y, x_{e-1})}$ must be symmetric. Consequently, the Hilbert function of $\dfrac{R(e,m,n)[y]}{\operatorname{in} _{<'}\left(I_{\overline{S(e,m,n)}}+(y, x_{e-1})\right)}$ must also be symmetric. Note that since $\mathbb{K}[\overline{S(e,m,n)}]$ is Cohen--Macaulay, by \Cref{thm:CM-criterion}, $y, x_{e-1}$ do not divide any of minimal generators of $\operatorname{in}_{<'}\left(I_{\overline{S(e,m,n)}}\right)$. Therefore, we have the equality 
$$\operatorname{in}_{<'}\left(I_{\overline{S(e,m,n)}}+(y, x_{e-1})\right) =\operatorname{in}_{<'}\left(I_{\overline{S(e,m,n)}}\right)+ (y, x_{e-1}).$$ Also, by \cite[Lemma 2.1]{Herzog-Stamate-CM}, we have $\operatorname{in}_{<'}\left(I_{\overline{S(e,m,n)}}\right) = \operatorname{in}_<\left(I_{{S(e,m,n)}}\right) R(e,m,n)[y] $. Thus, the Hilbert function of $\dfrac{R(e,m,n)}{\operatorname{in}_<\left(I_{S(e,m,n)}\right)+(x_{e-1})} \eqqcolon T$ must be symmetric. By \Cref{thm:dim}, we have 
$$ \dim_{\mathbb{K}}\left(\dfrac{R(e,m,n)}{\operatorname{in}_<\left(I_{S(e,m,n)}\right)+(x_{e-1})}\right) = 2e-1.$$ 
 Let $T_i$ denote the degree $i$ component of $T$. Then, by \Cref{rem:Surviving-monomials-in-degrees-2-and-3}, we have $T_i=0$ for all $i\geq 4$. Observe that $\dim_{\mathbb{K}}(T_0)=1$ and $\dim(T_1)=e-3$. Hence, we get $\dim_{\mathbb{K}}(T_2)+\dim_{\mathbb{K}}(T_3)=(2e-1)-1-(e-3)=e+1$. \\
 If $T_3=0$, then the symmetry of the Hilbert function forces $\dim_{\mathbb{K}}(T_2)=\dim_{\mathbb{K}}(T_0)=1$, and hence $\dim_{\mathbb{K}}(T_2)+\dim_{\mathbb{K}}(T_3)=1\neq e+1$, a contradiction. On the other hand, if $T_3 \neq 0$, then we have $\dim_{\mathbb{K}}(T_2)=\dim_{\mathbb{K}}(T_1)=e-3$ and $\dim_{\mathbb{K}}(T_3)=\dim_{\mathbb{K}}(T_0)=1$. This again leads to the contradiction that $\dim_{\mathbb{K}}(T_2)+\dim_{\mathbb{K}}(T_3)=e-2\neq e+1$. This shows that if $e \geq 10$, then $\mathbb{K}[\overline{S(e,m,n)}]$ is not Gorenstein.
\end{proof}

Finally, we compute the Castelnuovo--Mumford regularity of $\mathbb{K}[\overline{S(e, m, n)}]$ when it is Cohen--Macaulay. 
We first recall the concept. Let $Q$  be a standard graded polynomial ring over a field $\mathbb{K}$, and $M$ a finitely generated graded $Q$-module. The \emph{Castelnuovo--Mumford regularity} of $M$ as an $Q$-module, denoted by $\reg_Q(M)$, is given by
\[
\reg_Q (M) \coloneqq \max\{j-i\mid \operatorname{Tor}_i(M,\mathbb{K})_j\neq 0\}.
\]

\begin{remark}\label{rmk:reg}    Assume that $\mathbb{K}[\overline{S(e, m, n)}]$ is Cohen--Macaulay for some integers $e,m,n$. Since $\{y, x_{e-1}\}$ is a linear system of parameters on $\mathbb{K}[\overline{S(e, m, n)}]$, it is also a linear regular sequence. Therefore, it is straightforward (see e.g., \cite[Proposition~1.1.5]{BH93}) that we have the equality 
    $$\reg_{R(e,m,n)[y]}\left(\mathbb K[\overline{S(e,m,n)}]\right) = \reg_{\frac{R(e,m,n)[y]}{(y,x_{e-1})}}\left( \dfrac{\mathbb K[\overline{S(e,m,n)}]}{(y,x_{e-1})}\right).$$
\end{remark}

From \Cref{rem:Surviving-monomials-in-degrees-2-and-3} and \Cref{rmk:reg}, we obtain the Castelnuovo--Mumford regularity of $\mathbb{K}[\overline{S(e, m, n)}]$ over $R(e,m,n)[y]$. 

\begin{theorem}
Let $\mathbb{K}[\overline{S(e, m, n)}]$ be Cohen--Macaulay. Then


\[
\reg_{R(e,m,n)[y]}\!\left( \mathbb{K}[\overline{S(e, m, n)}] \right)
=
\begin{cases}
6 & \text{if } (e,m,n) = (4,1,2) \\[4pt]
4 & \text{if } (e,m,n) \in \{ (5,1,2), (5,1,3), (5,2,3) \} \\[4pt]
3 & \text{if } e \ge 6 \text{ and } m = 1 \\[4pt]
3 & \text{if } e \ge 6 \text{ and } (m,n) \in \{ (2,3), (e-3,e-2) (e-4,e-2)\} \\[4pt]
3 & \text{if } e \ge 6  \text{ and } (m,e-2);~2 \le m \le e-5 \\
2 & \text{otherwise.}
\end{cases}
\]

\end{theorem}
\begin{proof}
A direct computation using \texttt{Macaulay2}~\cite{M2} shows that that the result holds when $4 \leq e \leq 9$. Thus, we assume that $e \geq 10$.  By \Cref{rmk:reg}, we have 
  $$\reg_{R(e,m,n)[y]}\left(\mathbb K[\overline{S(e,m,n)}]\right) = \reg_{\frac{R(e,m,n)[y]}{(x_{e-1},y)}}\left( \dfrac{R(e,m,n)[y]}{I_{\overline{S(e,m,n)}}+(y,x_{e-1})}\right).$$

  Since  $\dfrac{R(e,m,n)[y]}{I_{\overline{S(e,m,n)}}+(y,x_{e-1})}$ is Artinian, by \cite[Exercise 20.18]{Eis-book}, $\reg_{\frac{R(e,m,n)[y]}{(x_{e-1},y)}}\left( \dfrac{R(e,m,n)[y]}{I_{\overline{S(e,m,n)}}+(y,x_{e-1})}\right)$ equals the largest degree of a nonzero element in $\dfrac{R(e,m,n)[y]}{I_{\overline{S(e,m,n)}}+(y,x_{e-1})}$. Since $\mathbb{K}\left[\overline{S(e, m, n)}\right]$ is Cohen--Macaulay, as observed earlier, we have $$\operatorname{in}_{<'}\left(I_{\overline{S(e,m,n)}}+(y, x_{e-1})\right) =\operatorname{in}_{<'}\left(I_{\overline{S(e,m,n)}}\right)+ (y, x_{e-1}) \text{\  \ and \ \ } \operatorname{in}_{<'}(I_{\overline{S(e,m,n)}}) = \operatorname{in}_<(I_{{S(e,m,n)}}) R(e,m,n)[y] .$$ 
  Therefore, $\reg_{R(e,m,n)[y]}\left(\mathbb K[\overline{S(e,m,n)}]\right)$ equals the largest degree of a nonzero element in $\dfrac{R(e,m,n)}{\operatorname{in}_<\left(I_{{S(e,m,n)}}\right)+(x_{e-1})}$. The result then follows directly from \Cref{rem:Surviving-monomials-in-degrees-2-and-3}.
\end{proof}

\section*{Appendix: Some \texttt{Macaulay2} codes}

In this appendix we provide a few \texttt{Macaulay2} \cite{M2} codes that we utilized to study these projective monomial curves. To compute the defining ideal $I(e,m,n)$ of $\mathbb{Q}[S(e,m,n)]$ for specific values of the integers $e,m,n$, we can use the following codes together with the package \texttt{NumericalSemigroups}, as follows.

{\color{blue}
\small

\setstretch{.67}
\begin{lstlisting}

i1 : loadPackage "NumericalSemigroups";

i2 : e = 5; m = 1; n = 2;

i3 : delete(e + n, delete(e + m, e..(2*e - 1)));

i3 : mingens semigroupIdeal (toList oo, "BaseField" => QQ)

o3 = $(\ x_0^2x_3-x_4^2\quad x_3^3-x_0^3x_4\quad x_0^5-x_3^2x_4\ )$

o3 : Matrix $(\mathbb{Q}[x_0,x_3..x_4])^1\leftarrow (\mathbb{Q}[x_0,x_3..x_4])^3$
  
\end{lstlisting}
}
\medskip

Similarly, we provide the codes to obtain the defining ideal $I_{\overline{S(e,m,n)}}$ of $\mathbb Q [\overline{S(e,m,n)}]$ (and the  Castelnuovo--Mumford regularity of the quotient ring). There is also a specific example below.

{\color{blue}
\small

\setstretch{.67}
\begin{lstlisting}

i1 : sallyMap = (e, m, n) -> (
    R := (   
        T := QQ[x_0..x_(e-1),y];
        vars := gens T;
        keep := delete(m, toList(0..e));
        keep = delete(n, keep);
        QQ[(apply(keep, k -> vars#k))]
        );
    S := QQ[t_1, t_2];
    imageList := {};
    -- images of x_i
    for i from 0 to e-1 do (
        if not member(i, {m,n}) then (
            imageList = append(imageList, t_1^(e+i) * t_2^(2*e-1-(e+i)))
            )
        );
    -- image of y
    imageList = append(imageList, t_2^(2*e-1));
    -- Construct the ring map
    map(S, R, imageList)
    );

i2 : ker sallyMap(5,1,2)

o2 = ideal($x_0^2x_3-x_4^2y,\ x_0^3x_4-x_3^3y,\ x_3^4-x_0x_4^3,\ x_0^5-x_3^2x_4y^2$) 

o2 : Ideal of $\mathbb{Q}[x_0,\ x_3..x_4, \ y]$

i3 : regularity ((ring oo)^1/ oo)

o3 = 4
\end{lstlisting}
}
\medskip

\bibliographystyle{abbrv}
\bibliography{refs}
 
\end{document}